\documentclass[10pt]{amsart}
\usepackage{amsfonts}
\usepackage{indentfirst,amsthm,bm,amsmath}
\usepackage{hyperref}

\newtheorem{theorem}{Theorem}[section]
\newtheorem{corollary}[theorem]{Corollary}
\newtheorem{lemma}[theorem]{Lemma}

\newtheorem{remark}[theorem]{Remark}
\newtheorem{example}[theorem]{Example}

\begin{document}
\title[]{Meromorphic solutions of delay differential equations related to logistic type and generalizations}

\author[Ling Xu]{Ling Xu}
\address[Ling Xu]{$^{1}$ School of Life Sciences, Jiangxi Science and Technology Normal University, Jiangxi 330013, P.R. China; $^{2}$Department of Mathematics, Nanchang University, Jiangxi 330031, P. R. China; }
\email{xuling-jxstnu@foxmail.com}

\author[Tingbin Cao]{Tingbin Cao}
\address[Tingbin Cao(the corresponding author)]{Department of Mathematics, Nanchang University,  Jiangxi 330031, P. R. China}
\email{tbcao@ncu.edu.cn}
\thanks{The work is supported by the National Natural Science Foundation of China (\#11871260, \#11461042).}






\begin{abstract}Let $\{b_{j}\}_{j=1}^{k}$ be meromorphic functions, and let $w$ be admissible meromorphic solutions of delay differential equation $$w'(z)=w(z)\left[\frac{P(z, w(z))}{Q(z,w(z))}+\sum_{j=1}^{k}b_{j}(z)w(z-c_{j})\right]$$ with distinct delays $c_{1}, \ldots, c_{k}\in\mathbb{C}\setminus\{0\},$  where the two nonzero polynomials $P(z, w(z))$ and $Q(z, w(z))$ in $w$ with meromorphic coefficients are prime each other. We obtain that if $\limsup_{r\rightarrow\infty}\frac{\log T(r, w)}{r}=0,$ then $$\deg_{w}(P/Q)\leq k+2.$$ Furthermore, if $Q(z, w(z))$ has at least one nonzero root, then $\deg_{w}(P)=\deg_{w}(Q)+1\leq k+2;$ if all roots of $Q(z, w(z))$ are nonzero, then $\deg_{w}(P)=\deg_{w}(Q)+1\leq k+1;$  if $\deg_{w}(Q)=0,$  then $\deg_{w}(P)\leq 1.$\par

In particular, whenever $\deg_{w}(Q)=0$ and $\deg_{w}(P)\leq 1$ and without the growth condition, any admissible meromorphic solution of the above delay differential equation (called Lenhart-Travis' type logistic delay differential equation) with reduced form can not be an entire function $w$ satisfying $\overline{N}(r, \frac{1}{w})=O(N(r, \frac{1}{w}));$ while if all coefficients are rational functions, then the condition  $\overline{N}(r, \frac{1}{w})=O(N(r, \frac{1}{w}))$ can be omitted. Furthermore, any admissible meromorphic solution of the logistic delay differential equation (that is, for the simplest special case where $k=1$ and $\deg_{w}(P/Q)=0$ ) satisfies that $N(r,w)$ and $T(r, w)$ have the same growth category. Some examples support our results.
\end{abstract}

\maketitle

\section{Introduction and main results}
The logistic delay differential equations is one of the most important delay differential equations which are primarily taken from the biological sciences literature such as population biology, physiology, epidemiology, and neural networks. In 1838, Verhulst \cite{verhulst} investigated the growth of a single population and proposed the famous logistic equation $w'(t)=w(t)\left[a-bw(t)\right].$ It assumes that population density negatively affects the per capita growth rate in terms of $\frac{w'(t)}{w(t)}=a-bw(t)$ due to environmental degradation. In 1948, Hutchinson \cite{hutchinson} pointed out that negative effects that high population densities have on the environment influence birth rates at later times due to developmental and maturation delays. This led him to propose the delayed logistic differential equation $w'(t)=w(t)\left[a-bw(t-r)\right]$ where $a, b, r>0$ and $r$ is called the delay. In 1986, Lenhart and Travis \cite{lenhart-travis} studied the widely logistic model of population dynamics $$w'(t)=w(t)\left[a_{0}-a_{1}x(t)-\sum_{j=1}^{k}b_{j}w(t-r_{j})\right]$$ where $a_{0}, a_{1}, b_{j}>0$ and $r_{j}$ are the distinct delays. It is known that both theory and applications of delay differential equations require a bit more mathematical maturity than its ordinary differential equations counterparts, in which primarily,  the theory of complex analysis plays a large role. For the background, we refer to \cite{smith}. \par

As a generalization of the logistic equation $w'(z)=w(z)\left[b(z)+a(z)w(z)\right]$ with coefficients $a, b$ meromorphic on the complex plane, the Riccati differential equation $$w'=a_{0}(z)+a_{1}(z)w+a_{2}w^{2}$$ was investigated by H. Wittich \cite{wittich} in 1960 by using the method of Nevanlinna theory in complex analysis. Later on, many mathematician deeply studied more general differential equation of the Yosida-Malmquist type $$w^{'}(z)=R(z, f(z)),$$ where $R(z, f(z))$ is a rational function of $f$ in $z$ (see \cite{yosida}, \cite{bank-kaufman} or \cite{laine}). Thus there arises an interesting topic on the study of solutions of delay differential equations of logistic type and generalization in the viewpoint of Nevanlinna theory.\par

In 2017, Halburd and Korhonen \cite{halburd-korhonen} obtained the following theorem which is motivated by G. R. W. Quispel, H. W. Capel, and R. Sahadevan \cite{quispel-capel-sahadevan}. Recently, R. R. Zhang and Z. B. Huang \cite{zhang-huang}, K. Liu and C. J. Song \cite{liu-song} also studied this topic.\par

\begin{theorem}\cite[Theorem 1.1]{halburd-korhonen} \label{TA} Let $w(z)$ be a non-rational meromorphic solution of
\begin{eqnarray}\label{EA}w(z+1)-w(z-1)+a(z)\frac{w^{'}(z)}{w(z)}=R(z, w(z))=\frac{P(z, w(z))}{Q(z, w(z))},
\end{eqnarray} where $a(z)$ is rational, $P(z, w(z))$ is a polynomial in $w$ having rational coefficients in $z,$ and $Q(z, w(z))$ is a polynomial in $w(z)$ with roots that are nonzero rational functions of $z$ and not roots of $P(z, w).$ If the hyperorder of $w(z)$ is less than one, then $$\deg_{w}(P)=\deg_{w}(Q)+1\leq 3 \,\,\mbox{or} \,\, \deg_{w}(R)\leq 1.$$
\end{theorem}

The main purpose of this paper is to study the admissible meromorphic solutions of a generalized delay differential equation combining the logistic type delay differential equations due to Lenhart and Travis \cite{lenhart-travis} and  the delay differential equation due to Halburd and Korhonen \cite{halburd-korhonen}. Before stating our main theorems, let us introduce briefly some basic definitions and notations of Nevanlinna theory for a meromorphic function $f$ in the complex plane $\mathbb{C}.$  For every real number $x\geq 0,$ we define $\log^{+}x:=\max\{0, \log x\}.$ Assume that $n(r, f)$ counts the number of the poles of $f$ in $|z|\leq r$ (counting multiplicity), if ignoring multiplicity, then denote it by $\overline{n}(r, f).$ The Nevanlinna's characteristic function of $f$ is defined by $$T(r,f):=m(r,f)+N(r,f),$$ where
$$N(r,f):=\int_{0}^{r}\frac{n(t,f)-n(0,f)}{t}dt+n(0,f)\log r$$ is called the counting function of poles of $f$ and
$$m(r,f):=\frac{1}{2\pi}\int_{0}^{2\pi}\log^{+}\left|f(re^{i\theta})\right|d\theta$$ is called the proximity function of $f.$ The hyperorder of $f$ is defined by $$\rho_{2}(f):=\limsup_{r\rightarrow\infty}\frac{\log^{+}\log^{+} T(r, f)}{\log r}.$$
The first main theorem in Nevanlinna theory states that $$T(r, \frac{1}{f-a})=T(r,f)+o(T(r, f))$$ holds for any value $a\in\mathbb{C}.$ For more notations and definitions of the Nevanlinna theory, refer to \cite{hayman-1}. Recall that a meromorphic function $g$ is said to be a \emph{small function} with respect to another given meromorphic function $f,$ provided that $T(r, g)=o(T(r, f)),$ possibly outside of a set with finite logarithmic measure. We denote it by $g\in S(f)$ sometimes. For instances, constants are small with respect to nonconstant entire or meromorphic functions, and polynomial (or rational) functions are small with respect to transcendental entire (or meromorphic) functions; the entire function $e^{z}$ is a small function with respect to the meromorphic function $\frac{e^{e^{z}}}{z-1}.$ Obviously, all the transcendental meromorphic solutions of \eqref{EA} with rational coefficients are admissible solutions. More general, if all of the coefficient functions are small functions with respect to a nontrivial meromorphic function solution $w$ of a general functional differential equation, then the solution $w$ is called an \emph{admissible solution}. We will introduce some other basic results of Nevanlinna theory if necessary. \par

Namely, we will study the meromorphic solutions to the delay differential equation of a more general form combining the Lenhart-Travis' logistic type and Halburd-Korhonen's type as follows:
\begin{eqnarray}\label{E2} w'(z)&=&w(z)\left[R(z, w(z))+\sum_{j=1}^{k}b_{j}(z)w(z-c_{j})\right]\\\nonumber
&=&w(z)\left[\frac{P(z, w(z))}{Q(z, w(z))}+\sum_{j=1}^{k}b_{j}(z)w(z-c_{j})\right]
\end{eqnarray} with distinct delays $c_{1}, \ldots, c_{k}\in\mathbb{C}^{*}=\mathbb{C}\setminus\{0\}.$ We obtain the following result, which is an improvement and extension of Theorem \ref{TA} for $k=2$ and \cite[Theorem 1.2]{song-liu-ma} for $k=1$ with rational coefficients. \par

\begin{theorem}\label{T3} Let $\{b_{j}\}_{j=1}^{k}$ be meromorphic functions, and let $w$ be an admissible meromorphic solution of the delay differential equation \eqref{E2}, where the two nonzero polynomials $P(z, w(z))$ and $Q(z, w(z))$ in $w$ with meromorphic coefficients are prime each other (that is, having no common factors). If $\limsup_{r\rightarrow\infty}\frac{\log T(r, w)}{r}=0,$ then $$\deg_{w}(R)\leq k+2,$$ where $\deg_{w}(R)=\max\{\deg_{w}(P), \deg_{w}(Q)\}.$ Furthermore,\par

(i). if $Q(z, w(z))$ has at least one nonzero root, then $$\deg_{w}(P)=\deg_{w}(Q)+1\leq k+2;$$\par

(ii). under the assumption of (i), assume further that all roots of $Q(z, w(z))$ are nonzero, then $$\deg_{w}(P)=\deg_{w}(Q)+1\leq k+1;$$\par

(iii). if $\deg_{w}(Q)=0$ (that is, $R(z, w(z))$ is degenerated to a polynomial in $w$), then $$\deg_{w}(P)=\deg_{w}(R)\leq 1.$$
\end{theorem}

Remark that the assumption $\limsup_{r\rightarrow\infty}\frac{\log T(r, w)}{r}=0$ is better than the condition of hyperorder strict less than one in Theorem \ref{TA}, based on the improvement of the difference version of logarithmic derivative lemma \cite{cao-xu, zheng-korhonen-2018}. We may call it minimal hypertype.\par

Below we give several examples to explain Theorem \ref{T3}. The first and second examples are given to show the inequality ``$\deg_{w}(P)=\deg_{w}(Q)+1=k+1$" is possible to get and thus the conclusion (ii) in Theorem \ref{T3} is sharp for both rational function solutions and transcendental meromorphic solutions.\par

\begin{example} It is easy to check that the rational function $w=\frac{1}{z}$ is an admissible solution of the delay differential equation $$w'(z)=w(z)\left[\frac{-w^{2}(z)+2w(z)}{w(z)-1}+w(z-1)
\right]$$ with constant coefficients. Let $P(z, w(z)):=-w^{2}(z)+2w(z)$ and $Q(z, w(z)):=w(z)-1.$ Then we have $\deg_{w}(P)=1\deg_{w}(Q)+1=2=k+1,$ where $k=1$ is the number of the delays. This means that the inequality ``$\deg_{w}(P)=\deg_{w}(Q)+1=k+1$" is possible to get and thus the conclusion of (ii) in Theorem \ref{T3} is sharp for admissible rational solutions.
\end{example}

\begin{example}\cite[Example 1.4]{song-liu-ma} It is easy to check that the meromorphic function $w(z)=\frac{1}{e^{z}+1}$ is an admissible solution of the delay differential equation $$w'(z)=w(z)\left(\frac{w^{2}(z)+\frac{2e}{1-e}w(z)-\frac{e}{1-e}}{w(z)+\frac{e}{1-e}}-w(z+1)\right).$$ Let $P(z, w(z)):=w^{2}(z)+\frac{2e}{1-e}w(z)-\frac{e}{1-e}$ and $Q(z, w(z)):=w(z)+\frac{e}{1-e}.$ Then we have $\deg_{w}(P)=\deg_{w}(Q)+1=2=k+1,$  where $k=1$ is the number of the delays. This implies that the inequality ``$\deg_{w}(P)=\deg_{w}(Q)+1=k+1$" is possible to get and thus the conclusion (ii) in Theorem \ref{T3} is sharp for admissible transcendental meromorphic solutions.\end{example}

The following example shows that it is necessary of the condition of minimal hypertype $\limsup_{r\rightarrow\infty}$ $\frac{\log T(r, w)}{r}=0$ in Theorem \ref{T3}.\par

\begin{example} Let $n\geq 4.$ It is easy to check that the entire function $w(z)=e^{e^{z}}$ is an admissible solution of the delay differential equation $$w'(z)=w(z)\left(e^{z}+w^{n}(z)-w(z+\log n)\right).$$ Since $T(r, w)=\frac{r}{\pi},$ we have $\limsup_{r\rightarrow\infty}\frac{\log T(r, w)}{r}=\frac{1}{\pi}>0.$ Let $Q(z, w(z)):=1$ and $P(z, w(z))=R(z, w(z))=e^{z}+w^{n}(z).$ Then we have $n=\deg_{w}(R)>k+2=3$ instead of $\deg_{w}(R)\leq k+2,$ where $k=1$ is the number of delays. This implies that in Theorem \ref{T3}, the assumption of $\limsup_{r\rightarrow\infty}$ $\frac{\log T(r, w)}{r}=0$ of the growth of solutions is necessary.\end{example}

The following example implies that there is some admissible meromorphic solutions such that the inequality $\deg_{w}(P)=\deg_{w}(R)=1$ holds, under the assumption of the case (iii) of Theorem \ref{T3}.\par

\begin{example}It is not difficult to deduce that the meromorphic function $w(z)=\frac{1}{e^{2\pi iz}-1}$ is an admissible solution of the delay differential equation
$$w'(z)=w(z)\left[-2\pi iz+b(z)w(z)+\sum_{j=1}^{k}b_{j}(z)w(z+k)\right],$$  where the polynomials $b_{0}, b_{1}, \ldots, b_{k}$ satisfy $b(z)+b_{1}(z)+\ldots+b_{k}\equiv 2\pi iz.$ Let $P(z, w(z)):=-2\pi iz+b(z)w(z)$ and $Q(z, w(z)):=1.$ Then we get $\deg_{w}(P)=\deg_{w}(R)=1.$ This implies that the inequality `$\deg_{w}(P)=\deg_{w}(R)=1$" is possible to get and thus the  conclusion of (iii) in Theorem \ref{T3} is sharp for admissible meromorphic solutions.
\end{example}

Next, we will deeply study the logistic delay differential equations with coefficients meromorphic on the complex plane $\mathbb{C}$ of the Lenhart-Travis' type  (that is, $\deg_{w}Q=0$ and $\deg_{w}P\leq 1$ in the equation \eqref{E2}) \begin{equation}\label{E1.1} w'(z)=w(z)\left[a_{0}(z)+a_{1}(z)w (z)+\sum_{j=1}^{k}b_{j}(z)w(z-c_{j})\right]\end{equation} where the delays $c_{j}$ are distinct constants of $\mathbb{C}^{*}=\mathbb{C}\setminus\{0\},$  the coefficients $a_{0}, a_{1}$ and $b_{1}, \ldots, b_{k}$ are meromorphic functions, and each of $b_{1}, \ldots, b_{k}$ is not identically equal to zero.\par

Without loss of generality, we may assume that
\begin{itemize}
  \item There are no any $m\in\{1, \ldots, k\}$ such that either $w(z-c_{m})\equiv \alpha(z) w(z)$ or $w(z-c_{m})\equiv\alpha(z) w(z-c_{n}),$ where $T(r,\alpha)=o(T(r, w))$ and $n\in\{1, \ldots, k\}\setminus\{m\}.$
\end{itemize}
In fact, if this case happened, then the equation \eqref{E1.1} (or \eqref{E2}) could be reduced to the delay differential equation with just fewer delays than $k.$ We may call the equation \eqref{E1.1}  (or \eqref{E2})  with \emph{reduced form}, provided that all nontrivial solutions of \eqref{E1.1}  (or \eqref{E2})  satisfy this assumption.\par

By introducing the concept of the reduced form for the first time, we can focus on the logistic delay differential equations \eqref{E1.1} and obtain that ``most" of admissible meromorphic solutions should have a pole at least. It remains open whether the condition of $\overline{N}(r, \frac{1}{w})=O(N(r, \frac{1}{w}))$(or assume that zeros of admissible entire solutions are of uniformly bounded multiplicities, which has ever appeared in \cite{cao-xu-chen}) can be deleted or not.\par

\begin{theorem}\label{T1} Any admissible meromorphic solution $w$ of the Lenhart-Travis' type logistic delay differential equation \eqref{E1.1} with reduced form can not be an entire function satisfying $\overline{N}(r, \frac{1}{w})=O(N(r, \frac{1}{w})).$
\end{theorem}

In other words, Theorem \ref{T1} shows that all admissible entire solutions $w$ of the equation \eqref{E1.1} with $\overline{N}(r, \frac{1}{w})=O(N(r, \frac{1}{w}))$ must satisfy the first order difference equations $$w(z-c_{j})\equiv \alpha(z) w(z)$$ or $$w(z-c_{m})\equiv\alpha(z) w(z-c_{n}),$$ where $j\in\{1, \ldots, k\},$ $1\leq m<n\leq k,$ and $\alpha$ is small with respect to $w.$ This leads people to only need study the admissible entire solution of the difference equation $$w(z-c)+a(z)w(z)=0.$$ We note that this kind of first order difference equations has been investigated deeply by Z. X. Chen \cite{chen, chen-1, chen-book} and others. In particular, if the equation \eqref{E1.1} for $k=1$ which thus reduces to
\begin{equation}\label{Ee}
    w'(z)=w(z)\left[a_{0}(z)+a_{1}(z)w(z)+b_{1}(z)w(z-c_{1})\right]
\end{equation} is not a reduced form, then there exists a small meromorphic function $a(z)$ with respect to $w(z)$ such that $w(z-c_{1})\equiv a(z)w(z).$ Then this equation becomes
\begin{equation}\label{EEE}w'(z)=\left[a_{1}(z)+a(z)b_{1}(z)\right]w(z)^{2}+a_{0}(z)w(z).\end{equation}  It follows from \cite[Theorem 9.1.12]{laine} that any admissible meromorphic solution of the Riccati differential equation  $w'=A_{0}(z)+A_{1}(z)w+A_{2}(z)w^{2},$ where $A_{2}(z)\not\equiv 0$  satisfies $\delta(f, \infty)=1-\limsup_{r\rightarrow\infty}\frac{N(r, w)}{T(r, w)}=0.$ Hence the equation \eqref{EEE} having an admissible entire solution implies that $a_{1}(z)+a(z)b_{1}(z)
\equiv 0.$ Therefore, we have the following result.\par

\begin{theorem}\label{T5}If the logistic delay differential equation \eqref{Ee} which is not a reduced form has an admissible entire solution, then all admissible meromorphic solutions of \eqref{Ee} must satisfy both $w'(z)\equiv a_{0}(z)w(z)$ and $w(z-c_{1})\equiv -\frac{a_{1}(z)}{b_{1}(z)}w(z).$\end{theorem}

If the coefficient functions $a_{0}, a_{1}$ and $b_{1}$ are given, then Theorem \ref{T5} implies that it is possible to solve all admissible meromorphic solutions of equation \eqref{Ee}. Below we give an example to show how to obtain all admissible meromorphic solutions by Theorem \ref{T5}.\par

\begin{example} It is easy to check that the transcendental entire function $w=ze^{z}$ is an admissible entire solution of the delay differential equation \begin{equation}\label{EE} w'(z)=w(z)\left[\frac{z+1}{z}+\frac{1-z}{ez}w(z)+w(z-1)
\right].\end{equation} Let $a_{0}(z):=\frac{z+1}{z},$ $a_{1}(z):=\frac{1-z}{ez}$ and $b_{1}(z):=1.$ Since $w(z)=ze^{z}$ only has a simple zero $z=0,$ it satisfies $\overline{N}(r, \frac{1}{w})=N(r, \frac{1}{w}).$ It shows that the condition of reduced form in Theorem \ref{T1} is necessary, and this implies that the equation \eqref{EE} should not be a reduced form. Then it follows from Theorem \ref{T5} that all admissible meromorphic solutions of \eqref{EE} satisfy both \begin{eqnarray} \label{EE1} w'(z)\equiv \frac{z+1}{z}w(z)\end{eqnarray} and
\begin{eqnarray}\label{EE2}
w(z-1)\equiv \frac{z-1}{ez}w(z).\end{eqnarray}
By computing the multiplicities on both sides of equation \eqref{EE1}, it is easy to get that there does not exist  any $z_{0}\in\mathbb{C}\setminus\{0\}$ such that $z_{0}$ could be either zeros or poles of $w(z).$ This means that only $z=0$ may be zeros or poles of $w(z).$ Furthermore, it follows from equation \eqref{EE1} that $z=0$ must be a zero or pole of $w(z).$  If $z=0$ is a pole of $w(z),$ then equation \eqref{EE2} gives that $z=0$ must be a pole of $w(z-1),$ and thus $z=-1$ is a pole of $w(z),$ a contradiction. If $z=0$ is a zero of $w(z),$ then it must be a simple zero. Otherwise, it would follow from \eqref{EE2} that $z=0$ must be a zero of $w(z-1),$ and thus $z=-1$ is a zero of $w(z),$ a contradiction. Hence, we get that $z=0$ must be a simple zero of $w(z).$ We may assume that $w(z)=ze^{g(z)}$ where $g(z)$ is an entire function. Submitting this into equation \eqref{EE1} gives that $g'(z)\equiv 1$ and thus $g(z)=z+K$ where $K$ is a constant.\end{example}

Hence we have an interesting corollary as follows.\par

\begin{corollary}All admissible meromorphic solutions $w(z)$ of the logistic delay differential equation \eqref{EE} must be the form $w(z)=ze^{z+K}$ where $K$ is a constant.
\end{corollary}

We remark that the conclusion (iii) of Theorem \ref{T3} shows that when $R(z, w(z))$ is degenerated to a polynomial $P(z, w(z))$ with meromorphic coefficients, each admissible meromorphic solution of the general form of equation \eqref{E2} is generated to the admissible meromorphic solutions of the logistic delay differential equation \eqref{E1.1}, and thus Theorem \ref{T3} together with Theorem \ref{T1} give the following corollary.\par

\begin{corollary} Let $\{b_{j}\}_{j=1}^{k}$ be meromorphic functions, and let $w$ be an admissible meromorphic solution of the delay differential equation \eqref{E2} with reduced form, where the two nonzero polynomials $P(z, w(z))$ and $Q(z, w(z))$ in $w$ with meromorphic coefficients are prime each other (that is, having no common factors). If $\limsup_{r\rightarrow\infty}\frac{\log T(r, w)}{r}=0$ and $\deg_{w}(Q)=0,$  then $w$ can not be entire function satisfying $\overline{N}(r, \frac{1}{w})=O(N(r, \frac{1}{w})).$
\end{corollary}

Moreover, it is interesting that if all coefficients of the delay differential equation \eqref{E1.1} are rational functions, then from the proof of Theorem \ref{T1}, one can easily see that any admissible entire function $w$ must have at most finitely many zeros. Thus in this case, the proof of Theorem \ref{T1} implies the following results which does not need the condition  $\overline{N}(r, \frac{1}{w})=O(N(r, \frac{1}{w})).$ \par

\begin{theorem} Let $a_{0},$ $a_{1},$ $b_{1}, \ldots, b_{k}$ be rational functions. Then any transcendental meromorphic solution of the logistic delay differential equation \eqref{E1.1} with reduced form  has at least one pole.
\end{theorem}

\begin{corollary} Let $\{b_{j}\}_{j=1}^{k}$ be rational functions, and let $w$ be an transcendental meromorphic solution of the delay differential equation \eqref{E2} with reduced form, where the two nonzero polynomials $P(z, w(z))$ and $Q(z, w(z))$ in $w$ with rational coefficients are prime each other (that is, having no common factors). If $\limsup_{r\rightarrow\infty}\frac{\log T(r, w)}{r}=0$ and $\deg_{w}(Q)=0,$ then $w$ has at least one pole.
\end{corollary}

At last, we consider the simplest case for $k=1$ and $\deg_{w}R=0$ in the delay differential equation \eqref{EA}, called the logistic delay differential equation \begin{equation}\label{E1.4} w'(z)=w(z)\left[a(z)+b(z)w(z-c)\right].\end{equation} We show that without the additional condition of reduced form in this case, any admissible meromorphic solution of the equation \eqref{E1.4} has at least one pole and  satisfies that $N(r, w)$ and $T(r, w)$ have the same growth category. This result improves and extends a recent result due to Song-Liu-Ma \cite[Theorem 1.7]{song-liu-ma}.\par

\begin{theorem}\label{T2} Let $c\in\mathbb{C}^{*},$  let $a$ and $b(\not\equiv 0)$ be two meromorphic functions. Then any admissible meromorphic solution $w$ of the logistic delay differential equation \eqref{E1.4} satisfies that $N(r, w)$ and $T(r, w)$ have the same growth category.
\end{theorem}

If $a$ and $b$ are rational functions, then we have the corollary.\par

\begin{corollary}\label{C12} Let $c\in\mathbb{C}^{*},$  let $a$ and $b(\not\equiv 0)$ be two rational functions. Then any transcendental meromorphic solution $w$ of the logistic delay differential equation \eqref{E1.4} satisfies that $N(r, w)$ and $T(r, w)$ have the same growth category.
\end{corollary}

The last example is given to show that the conclusions of both Theorem \ref{T2} and corollary \ref{C12} for the logistic differential equation  \eqref{E1.4} are really true. \par

\begin{example}\cite[Example 1.8]{song-liu-ma} It is easy to check that the meromorphic function $w(z)=\frac{1}{e^{2\pi iz}+1}$ is an admissible solution of the logistic delay differential equation $$w'(z)=w(z)(-2\pi i+2\pi i w(z+1)),$$ and $N(r, w)=T(r, w)=2r+O(1).$ Let $a(z):=-2\pi i$ and $b(z):=2\pi i.$ Then it satisfies the conclusions of Theorem \ref{T2} and Corollary \ref{C12}.\end{example}

Remark that the meromorphic function $w(z)=\frac{1}{e^{2\pi iz}+1}$ satisfies $w(z)\equiv w(z+1)$ in the above example, and thus it is also an admissible solution of the logistic differential equation $w'(z)=w(z)(-2\pi i+2\pi i w(z)).$ If the logistic delay differential equation \eqref{E1.4} is not reduced form, then it becomes the logistic differential equation which is a special case of the Riccati differential equation $w'=a_{0}(z)+a_{1}(z)w+a_{2}(z)w^{2},$ where $a_{2}(z)\not\equiv 0.$ It follows from \cite[Theorem 9.1.12]{laine} that any admissible meromorphic solution of the Riccati differential equation satisfies $\delta(f, \infty)=1-\limsup_{r\rightarrow\infty}\frac{N(r, w)}{T(r, w)}=0.$ This means that the conclusion of Theorem \ref{T2} is also true for logistic differential equation. \par

The remainder is the organization of this paper. In Section \ref{S2} we mainly give the proof of Theorem \ref{T3} by developing the iterative method to arbitrary $k$ distinct delays, which comes originally from Halburd and Korhonen \cite{halburd-korhonen} for two delays $\pm 1.$ More than two delays make the discussion more difficult and comprehensive. Theorem \ref{T1} and Theorem \ref{T2} are proved in Section \ref{S3} and Section \ref{S4} respectively. In order to consider without the condition of growth of solutions, we use the results on a meromorphic function and it difference due to Goldberg and Ostrovskii \cite{goldberg-ostrovskii}  instead of the difference version of logarithmic derivative lemma. \par

\section{Proof of Theorem \ref{T3}}\label{S2}
The difference version of logarithmic derivative lemma was established by Halburd-Korhonen\cite{halburd-korhonen-1} for hyperorder strictly less than one and Chiang-Feng \cite{chiang-feng} for finite order, independently. Here we introduce an improvement due to Zheng and Korhone \cite{zheng-korhonen-2018} recently, in which the growth of meromorphic function is extended to $\limsup_{r\rightarrow\infty}\frac{\log T(r, f)}{r}=0.$ The corresponding difference versions in several complex variables and in tropical geometry were obtained by  Cao-Xu \cite{cao-xu} and Cao-Zheng \cite{cao-zheng}, respectively. \par

\begin{lemma}\cite[Lemma 2.1]{zheng-korhonen-2018} \label{L0} Let $T(r)$ be a nondecreasing positive function in $[1, +\infty)$ and logarithmic convex with $T(r)\rightarrow+\infty (r\rightarrow+\infty).$  Assume that
\begin{equation}\label{E0}
\liminf_{r\rightarrow\infty} \frac{\log T(r)}{r}=0.
\end{equation} Set $\phi(r)=\max_{1\leq t\leq r}\{\frac{t}{\log T(t)}\}.$ Then given a constant $\delta\in(0, \frac{1}{2}),$ we have
\begin{equation*} T(r)\leq T(r+\phi^{\delta}(r))\leq \left(1+4\phi^{\delta-\frac{1}{2}}(r)\right)T(r), \,\, r\not\in E_{\delta},
\end{equation*}
where $E_{\delta}$ is a subset of $[1, +\infty)$ with the zero lower density. And $E_{\delta}$ has the zero upper density if \eqref{E0} holds for $\limsup.$
\end{lemma}

\begin{remark}Note that $\phi^{\delta}(r)\rightarrow\infty$ and $\phi^{\delta-\frac{1}{2}}(r)\rightarrow 0$ as $r\rightarrow \infty$ in Lemma \ref{L1}. Then for sufficiently large $r,$  we have $\phi^{\delta}(r)\geq h$ for any positive constant $h.$ Hence, $$T(r)\leq T(r+h)\leq T(r+\phi^{\delta}(r)) \leq (1+\varepsilon)T(r), \,\, r\not\in E,$$ where $E$ is a subset of $[1, +\infty)$ with the zero lower density.\end{remark} \par

\begin{lemma}[Difference version of logarithmic derivative lemma]\cite{zheng-korhonen-2018, cao-xu}\label{L2}
Let $f$ be a nonconstant meromorphic function and let $c\in\mathbb{C}\setminus\{0\}.$ If \begin{eqnarray}\label{E1}\limsup_{r\rightarrow\infty}\frac{\log T(r, f)}{r}=0,\end{eqnarray} then
\begin{eqnarray*}
m\left(r, \frac{f(z+c)}{f(z)}\right)+m\left(r, \frac{f(z)}{f(z+c)}\right)=o\left(T(r, f)\right)
\end{eqnarray*} for all $r\not\in E,$ where $E$ is a set with zero upper density measure $E,$ i.e., $$\overline{dens}E=\limsup_{r\rightarrow\infty}\frac{1}{r}\int_{E\cap [1,r]}dt=0.$$
\end{lemma}

The following lemma is due to Halburd and Korhonen \cite[Lemma 2.1]{halburd-korhonen}. Originally, they considered transcendental meromorphic solutions of equation $P(z, w)=0$ with rational coefficients. Here we consider admissible meromorphic solutions. Since it need only modify the proof by making use of the improvement of difference version of logarithmic derivative lemma (Lemma \ref{L2}) and the definition of small function, we omit the detail of its proof.\par

\begin{lemma} \label{L6} Let $w$ be an admissible meromorphic solution of the differential difference equation \begin{eqnarray*}&&P(z, w)\\&=&\sum_{l\in L}b_{l}(z)w(z)^{l_{0,0}}w(z+c_{1})^{l_{1, 0}}\cdots w(z+c_{\nu})^{l_{\nu, 0}}[w'(z)]^{l_{0, 1}}\cdots [w^{(\mu)}(z+c_{\nu})]^{l_{\nu, \mu}}\\&=&0\end{eqnarray*} where $c_{1}, \ldots, c_{\nu}$ are distinct complex constants, $L$ is a finite index set consisting of elements of the form $l=(l_{0, 0}, \ldots, l_{\nu, \mu})$ and the coefficients $b_{l}$ are meromorphic functions small with respect to $w$ for all $l\in L.$ Let $a_{1}, \ldots,
a_{k}$ be meromorphic functions small with respect to $w$ such that $P(z, a_{j})\not\equiv 0$ for all $j=1, \ldots, k.$ If there exist $s>0$ and $\tau\in(0, 1)$ such that
$$\sum_{j=1}^{k}n(r, \frac{1}{w-a_{j}})\leq k\tau n(r+s, w)+O(1),$$
then $\limsup_{r\rightarrow\infty}\frac{\log T(r, w)}{r}>0.$
\end{lemma}

\begin{proof}[Proof of Theorem \ref{T3}]Suppose that $w$ is an admissible meromorphic solution of equation \eqref{E2} and satisfies $\limsup_{r\rightarrow\infty}\frac{\log T(r, w)}{r}=0.$ Then by the first main theorem and Lemma \ref{L1}, it follows from \eqref{E2} that
\begin{eqnarray*}&&
T(r, R(z, w(z)))\\\nonumber&\leq& T(r, \frac{w'(z)}{w(z)})+\sum_{j=1}^{k}T(r, w(z-c_{j}))+\sum_{j=1}^{k}T(r, b_{j}(z))+O(1)\\\nonumber
&\leq& N(r, \frac{w'(z)}{w(z)})+\sum_{j=1}^{k}N(r, w(z-c_{j}))+\sum_{j=1}^{k}m(r, w(z-c_{j}))+o(T(r, w))\end{eqnarray*}for all $r\not\in E_1,$ where $E_1$ is a set with finite logarithmic measure (obviously, zero upper density measure). Noting that $N(r, \frac{w'(z)}{w(z)})\leq\overline{N}(r, w(z))+\overline{N}(r, \frac{1}{w(z)}),$ and combing with Lemma \ref{L2}. we then obtain
\begin{eqnarray*}&&T(r, R(z, w(z)))\\\nonumber&\leq& \overline{N}(r, w(z))+\overline{N}(r, \frac{1}{w(z)})+\sum_{j=1}^{k}N(r, w(z-c_{j}))+\sum_{j=1}^{k}m(r, \frac{w(z-c_{j})}{w(z)})\\\nonumber&&+m(r, w(z))+o(T(r, w))\\\nonumber
&\leq& \overline{N}(r, w(z))+\overline{N}(r, \frac{1}{w(z)})+\sum_{j=1}^{k}N(r, w(z-c_{j}))+m(r, w(z))+o(T(r, w))\end{eqnarray*}
for all $r\not\in E=E_{2}\cup E_{1}$ where $E_{2}$ is a set with zero upper density measure. This together with Lemma \ref{L6} gives
\begin{eqnarray}\label{E2.1}&&
T(r, R(z, w(z)))\\\nonumber
&\leq& \overline{N}(r, w(z))+\overline{N}(r, \frac{1}{w(z)})+\sum_{j=1}^{k}N(r+|c_{j}|, w(z))+m(r, w(z))+o(T(r, w))\\\nonumber
&\leq& \overline{N}(r, w(z))+\overline{N}(r, \frac{1}{w(z)})+kN(r, w(z))+m(r, w(z))+o(T(r, w))\\\nonumber
&\leq& (k+2)T(r, w(z))+o(T(r, w))
\end{eqnarray}for all $r\not\in E,$ where $E$ is a set with $\overline{dens}E=0.$ On the other hand, we get from the Valiron-Mohon'ko theorem (see for example in \cite{laine}) that
$$T(r, R(z, w(z)))=\deg_{w}(R(z, w(z)))T(r,w(z))+o(T(r, w)).$$Hence \eqref{E2.1} implies that
$$[\deg_{w}(R(z, w(z)))-(k+2)]T(r, w(z))\leq o(T(r, w))$$
for $r\not\in E.$ Hence $\deg_{w}(R(z, w(z)))\leq k+2.$\par

(i). Since $Q(z, w(z))$ has at least one non-zero root, and has no common factors with $P(z, w(z)),$ we may suppose that $Q(z, w(z))$ has just $n$ $(\deg_{w}(Q)\geq n\geq 1)$ distinct non-zero roots but not roots of $P(z, w(z)),$ say $d_{1}(z), \ldots, d_{n}(z),$ which are meromorphic functions small with respect to $w,$  such that the equation \eqref{E2} is rewritten as
\begin{eqnarray}\label{E2.2}
&&\frac{w'(z)}{w(z)}+\sum_{j=1}^{k}b_{j}(z)w(z-c_{j})\\\nonumber &=&R(z, w(z))=\frac{P(z, w(z))}{\widetilde{Q}(z, w(z))\prod_{j=1}^{n}[w(z)-d_{j}(z)]^{l_{j}}},
\end{eqnarray} where $(l_{1}, \ldots, l_{n})\in\mathbb{N}^{n},$  $\widetilde{Q}(z, w(z))$ is an irreducible polynomial in $w(z)$ having no common factors with $P(z, w(z)),$ $d_{1}, \ldots, d_{n}$ are not roots of  $P(z, w(z))$ and $\widetilde{Q}(z, w(z)).$ Obviously,  $\deg_{w}(\widetilde{Q})+\sum_{j=1}^{n}l_{j}=\deg_{w}(Q).$ Then none of $d_{1}, \ldots, d_{n}$ is an admissible solution of \eqref{E2.2}. \par

Assume that $z_{0}\in \mathbb{C}$ is a zero of $w(z)-d_{j}(z)$ $(j\in\{1, \ldots, n\}),$ say $w(z)-d_{1}(z),$ with multiplicity $t,$ but not a zero or a pole of any small meromorphic coefficients of \eqref{E2.2} and $P(z_{0}, w(z_{0}))\neq 0.$ This kind of points $z_{0}$ are called generic roots of $w-d_{1}$ with multiplicity $t.$ Since the coefficients of \eqref{E2.2} are all small with respect to $w,$ their counting functions are estimated into $o(T(r, w)).$ Hence we may only consider generic roots below.\par

By \eqref{E2.2}, we get that at least one of $w(z-c_{j})$ $(j\in \{1, \ldots, k\}),$ say $w(z-c_{1}),$ has a pole at $z=z_{0}$ with multiplicity at least $l_{1}t.$  Shifting the equation \eqref{E2.2} with $-c_{1}$ gives
\begin{eqnarray}\label{E2.3}&&
\frac{w'(z-c_{1})}{w(z-c_{1})}+\sum_{j=1}^{k}b_{j}(z-c_{1})w(z-c_{j}-c_{1})\\\nonumber&=&\frac{P(z-c_{1}, w(z-c_{1}))}{\widetilde{Q}(z-c_{1}, w(z-c_{1}))\prod_{j=1}^{n}\left[w(z-c_{1})-d_{j}(z-c_{1})\right]^{l_{j}}}.
\end{eqnarray} Then $z_{0}$ is a pole of $\frac{w'(z-c_{1})}{w(z-c_{1})}$ with simple multiplicity.\par

Case 1. Assume that $\deg_{w}(P)\leq \deg_{w}(Q).$ Then we discuss according to the following steps.\par

Step 1. Then \eqref{E2.3} implies that
at least one of $w(z-c_{j}-c_{1})$ $(j\in \{1, \ldots, k\}),$ say $w(z-c_{2}-c_{1}),$ has a pole at $z=z_{0}$ with multiplicity at least one. This implies $c_{1}+c_{2}\neq 0.$
Shifting the equation \eqref{E2.3} with $-c_{2}$ gives
\begin{eqnarray}\label{E2.4}&&
\frac{w'(z-c_{2}-c_{1})}{w(z-c_{2}-c_{1})}+\sum_{j=1}^{k}b_{j}(z-c_{2}-c_{1})w(z-c_{j}-c_{2}-c_{1})\\\nonumber&=&\frac{P(z-c_{2}-c_{1}, w(z-c_{2}-c_{1}))}{\widetilde{Q}(z-c_{2}-c_{1}, w(z-c_{2}-c_{1}))\prod_{j=1}^{n}\left[w(z-c_{2}-c_{1})-d_{j}(z-c_{2}-c_{1})\right]^{l_{j}}}.
\end{eqnarray} This implies that $z_{0}$ is a pole of $\frac{w'(z-c_{2}-c_{1})}{w(z-c_{2}-c_{1})}$ with simple multiplicity.\par

Step 2. Firstly, assume that there exists one term $m\in\{1, 2, \ldots, k\}$ such that  $c_{m}+c_{2}=0,$ and thus $w(z-c_{m}-c_{2}-c_{1})\equiv w(z-c_{1}).$ Then we stop the process and discuss as the following two subcases.\par

Subcase 1.1. Suppose that $l_{1}t>1.$ Then by \eqref{E2.4} we get that there exists at least another term $\hat{m}\in \{1, \ldots, k\}\setminus\{m\}$ such that $w(z-c_{\hat{m}}-c_{2}-c_{1})$ has a pole at $z=z_{0}$ with multiplicity at least $l_{1}t.$ Hence, even though $l_{j}t>1$ holds for each $w(z)-d_{j}(z)$ $(j\in\{1, \ldots, n\}),$ we obtain
\begin{eqnarray}\label{E3.1}
\sum_{j=1}^{n}n(r, \frac{1}{w-d_{j}})\leq K n(r+|c_{\hat{m}}|+|c_{2}|+|c_{1}|, w)+O(1).
\end{eqnarray}where $K:=\max_{j=1}^{n}\{\frac{t}{l_{j}t+1+l_{j}t}\}.$
\par

Subcase 1.2. Suppose that $l_{1}t=1.$ Then shifting the equation \eqref{E2.4} with $-c_{m}$ gives
\begin{eqnarray*}&&
\frac{w'(z-c_{m}-c_{2}-c_{1})}{w(z-c_{m}-c_{2}-c_{1})}+\sum_{j=1}^{k}b_{j}(z-c_{m}-c_{2}-c_{1})w(z-c_{j}-c_{m}-c_{2}-c_{1})\\\nonumber&=&\frac{1}{\widetilde{Q}(z-c_{m}-c_{2}-c_{1}, w(z-c_{m}-c_{2}-c_{1}))}\times\\&&\frac{P(z-c_{m}-c_{2}-c_{1}, w(z-c_{m}-c_{2}-c_{1}))}{\prod_{j=1}^{n}\left[w(z-c_{m}-c_{2}-c_{1})-d_{j}(z-c_{m}-c_{2}-c_{1})\right]^{l_{j}}}.
\end{eqnarray*} This equation is just the equation \eqref{E2.3}, since $c_{m}+c_{2}=0.$ We have obtained that $w(z-c_{2}-c_{1})$ has a pole at $z=z_{0}$ with multiplicity at least $l_{1}t=1.$ Hence, even though $l_{j}t=1$ holds for each $w(z)-d_{j}(z)$ $(j\in\{1, \ldots, n\}),$ we obtain \begin{eqnarray}\label{E3.2}
\sum_{j=1}^{n}n(r, \frac{1}{w-d_{j}})&\leq& \max_{j=1}^{n}\{\frac{t}{l_{j}t+1}\}n(r+|c_2|+|c_1|, w)+O(1)
\\\nonumber&=& \frac{1}{2}n(r+|c_2|+|c_1|, w)+O(1).
\end{eqnarray}\par

It may also possible that there exist some $w(z)-d_{j}(z)$ $(j\in (1, \ldots, n))$ whose zeros satisfy the Subcase 1.1 and others $w(z)-d_{j}(z)$ whose zeros satisfy the Subcase 1.2. Any way, we get from \eqref{E3.1} and \eqref{E3.2} that
\begin{eqnarray}\label{E3.2'} &&
\sum_{j=1}^{n}n(r, \frac{1}{w-d_{j}})\\\nonumber&\leq& \max\{K, \frac{1}{2}\} n(r+|c_{\hat{m}}|+|c_{2}|+|c_{1}|, w)+O(1).
\end{eqnarray}\par

Secondly, assume that there does not exist any term $m\in\{1, 2, \ldots, k\}$ such that  $c_{m}+c_{2}=0.$ Then we continue the process to the third step below.\par

Step 3. Now by \eqref{E2.4},  we  get that at least one of $w(z-c_{j}-c_{2}-c_{1})$ $(j\in \{1, \ldots, k\}),$ say $w(z-c_{3}-c_{2}-c_{1}),$ has a pole at $z=z_{0}$ with  multiplicity at least one. This implies $c_{3}+c_{2}+c_{1}\neq 0.$ Shifting the equation \eqref{E2.4} with $-c_{3}$ gives
\begin{eqnarray}\label{E2.6}\nonumber&&
\frac{w'(z-c_{3}-c_{2}-c_{1})}{w(z-c_{3}-c_{2}-c_{1})}+\sum_{j=1}^{k}b_{j}(z-c_{3}-c_{2}-c_{1})w(z-c_{j}-c_{3}-c_{2}-c_{1})\\\nonumber&=&
\frac{1}{\widetilde{Q}(z-c_{3}-c_{2}-c_{1},w(z-c_{3}-c_{2}-c_{1}))}\times\\&&
\frac{P(z-c_{3}-c_{2}-c_{1}, w(z-c_{3}-c_{2}-c_{1}))}{\prod_{j=1}^{n}\left[w(z-c_{3}-c_{2}-c_{1})-d_{j}(z-c_{3}-c_{2}-c_{1})\right]^{l_{j}}}.
\end{eqnarray} This implies that $z_{0}$ is a pole of $\frac{w'(z-c_{3}-c_{2}-c_{1})}{w(z-c_{3}-c_{2}-c_{1})}$ with simple multiplicity.\par

Now we give similar discussion as in the Step 2. Firstly, assume that there exists one term $m\in\{1, 2, \ldots, k\}$ such that  $c_{m}+c_{3}+c_{2}=0,$ and thus $w(z-c_{m}-c_{3}-c_{2}-c_{1})\equiv w(z-c_{1}).$ Then we stop the process and discuss as the following two subcases.\par

Subcase 1.1$^{*}$. Suppose that $l_{1}t>1.$ Then by \eqref{E2.6} we get that there exists at least another term $\hat{m}\in \{1, \ldots, k\}\setminus\{m\}$ such that $w(z-c_{\hat{m}}-c_{3}-c_{2}-c_{1})$ has a pole at $z=z_{0}$ with multiplicity at least $l_{1}t.$ Hence,even though $l_{j}t>1$ holds for each $w(z)-d_{j}(z)$ $(j\in\{1, \ldots, n\}),$ we obtain
\begin{eqnarray}\label{E3.3}
\sum_{j=1}^{n}n(r, \frac{1}{w-d_{j}})\leq K^{*} n(r+|c_{\hat{m}}|+|c_{3}|+|c_{2}|+|c_{1}|, w)+O(1).
\end{eqnarray}where $K^{*}:=\max_{j=1}^{n}\{\frac{t}{l_{j}t+1+1+l_{j}t}\}.$
\par

Subcase 1.2$^{*}$. Suppose that $l_{1}t=1.$ Then shifting the equation \eqref{E2.6} with $-c_{m}$ gives just the equation \eqref{E2.3}, since $c_{m}+c_{3}+c_{2}=0.$ We have obtained that $w(z-c_{2}-c_{1})$ has a pole at $z=z_{0}$ with multiplicity at least $l_{1}t=1.$ Hence even though $l_{j}t=1$ holds for each $w(z)-d_{j}(z)$ $(j\in\{1, \ldots, n\}),$ we obtain  \begin{eqnarray}\label{E3.4}
\sum_{j=1}^{n}n(r, \frac{1}{w-d_{j}})&\leq& \max_{j=1}^{n}\{\frac{t}{l_{j}t+1}\}n(r+|c_2|+|c_1|, w)+O(1)\\\nonumber&=& \frac{1}{2}n(r+|c_2|+|c_1|, w)+O(1).
\end{eqnarray}\par

It may also possible that there exist some $w(z)-d_{j}(z)$ $(j\in (1, \ldots, n))$ whose zeros satisfy the Subcase 1.1$^{*}$ and others $w(z)-d_{j}(z)$ whose zeros satisfy the Subcase 1.2$^{*}$. Any way, we get from \eqref{E3.3} and \eqref{E3.4} that
\begin{eqnarray}\label{E3.4'} &&
\sum_{j=1}^{n}n(r, \frac{1}{w-d_{j}})\\\nonumber&\leq& \max\{K^{*}, \frac{1}{2}\} n(r+|c_{\hat{m}}|+|c_{3}|+|c_{2}|+|c_{1}|, w)+O(1).
\end{eqnarray}\par

Secondly, assume that there exists one term $m\in\{1, 2, \ldots, k\}$ such that  $c_{m}+c_{3}=0,$ and thus $w(z-c_{m}-c_{3}-c_{2}-c_{1})=w(z-c_{2}-c_{1}).$ Then we stop the process and discuss as follows.  Shifting \eqref{E2.6} with $-c_{m}$ gives just the equation \eqref{E2.4}. We find that we come back the Step 2, and thus we obtained either one of \eqref{E3.1}, \eqref{E3.2}, \eqref{E3.2'}, \eqref{E3.3} , \eqref{E3.4} and \eqref{E3.4'}, or continue the following discussion.\par

Now assume that there does not exist one term $m\in\{1, 2, \ldots, k\}$ such that either $c_{m}+c_{3}+c_{2}=0$ or $c_{m}+c_{3}=0.$ Then we  continue the process to the Step 4 similarly as Step 3, and so on.  By this way with finite steps, finally we can always get that there exists one finite positive value $s$ depending on $|c_{1}|,\ldots, |c_{k}|,$ such that
\begin{eqnarray}\label{E2.8}\sum_{j=1}^{n} n(r, \frac{1}{w-d_{j}})\leq (n\cdot \tau)n(r+s, w)+O(1),\end{eqnarray}
where $\tau \in(0, 1)$ in Lemma \ref{L6}. Therefore, according to Lemma \ref{L6}, it follows that $\limsup_{r\rightarrow\infty}\frac{\log T(r, w)}{r}>0,$ which contradicts to the growth assumption of $w$ at the beginning of the proof.\par

Case 2. Assume that $\deg_{w}(P)> \deg_{w}(Q).$  Thus $$k+2\geq\deg_{w}(P)> \deg_{w}(Q)=\deg_{w}(\widetilde{Q})+\sum_{j=1}^{n}l_{j}\geq 1.$$

Assume that $\deg_{w}(P)-\deg_{w}(Q)\geq 2.$ We have assumed that $z_{0}\in \mathbb{C}$ is a generic root of $w(z)-d_{j}(z)$ $(j\in\{1, \ldots, k\}),$ say $w(z)-d_{1}(z),$ with multiplicity $t.$ Then again by \eqref{E2.2}  we get that at least one of $w(z-c_{j})$ $(j\in \{1, \ldots, k\}),$ say $w(z-c_{1}),$ has a pole at $z=z_{0}$ with multiplicity at least $l_{1}t;$  again by \eqref{E2.3} we get that $z_{0}$ is a pole of $$\frac{P(z-c_{1}, w(z-c_{1}))}{\widetilde{Q}(z-c_{1}, w(z-c_{1}))\prod_{j=1}^{n}\left[w(z-c_{1})-d_{j}(z-c_{1})\right]^{l_{j}}}$$ with multiplicity at least $l_{1}t(\deg_{w}(P)-\deg_{w}(Q))(\geq 2l_{1}t),$ and thus at least one of $w(z-c_{j}-c_{1})$ $(j\in \{1, \ldots, k\}),$ say $w(z-c_{2}-c_{1}),$ has a pole at $z=z_{0}$ with multiplicity at least $2l_{1}t;$
and again by \eqref{E2.4} we get that  $z_{0}$ is a pole of $$\frac{P(z-c_{2}-c_{1}, w(z-c_{2}-c_{1}))}{\widetilde{Q}(z-c_{2}-c_{1}, w(z-c_{2}-c_{1}) )\prod_{j=1}^{n}\left[w(z-c_{2}-c_{1})-d_{j}(z-c_{2}-c_{1})\right]^{l_{j}}}$$ with multiplicity at least $2l_{1}t(\deg_{w}(P)-\deg_{w}(Q))(\geq 4l_{1}t),$ and thus at least one of $w(z-c_{j}-c_{2}-c_{1})$ $(j\in \{1, \ldots, k\}),$ say $w(z-c_{3}-c_{2}-c_{1}),$ has a pole at $z=z_{0}$ with multiplicity at least $4l_{1}t.$  Again by \eqref{E2.6}, we get similarly that
at least one of $w(z-c_{j}-c_{3}-c_{2}-c_{1})$ $(j\in \{1, \ldots, k\}),$ say $w(z-c_{4}-c_{3}-c_{2}-c_{1}),$ has a pole at $z=z_{0}$ with multiplicity at least $8l_{1}t.$ Now shifting \eqref{E2.6} with $-c_{4}$ gives \begin{eqnarray*}\label{E2.9}&&
\frac{w'(z-\sum_{j=1}^{4}c_{j})}{w(z-\sum_{j=1}^{4}c_{j})}+\sum_{j=1}^{k}b_{j}(z-\sum_{j=1}^{4}c_{j})w(z-\sum_{j=1}^{4}c_{j})\\\nonumber&=&\frac{P(z-\sum_{j=1}^{4}c_{j}, w(z-\sum_{j=1}^{4}c_{j}))}{\widetilde{Q}(z-\sum_{j=1}^{4}c_{j}, w(z-\sum_{j=1}^{4}c_{j}))\prod_{j=1}^{n}\left[w(z-\sum_{j=1}^{4}c_{j})-d_{j}(z-\sum_{j=1}^{4}c_{j})\right]^{l_{j}}}.
\end{eqnarray*} Then we get similarly that
at least one of $w(z-c_{j}-c_{4}-c_{3}-c_{2}-c_{1})$ $(j\in \{1, \ldots, k\}),$ say $w(z-c_{5}-c_{4}-c_{3}-c_{2}-c_{1}),$ has a pole at $z=z_{0}$ with multiplicity at least $16l_{1}t.$ And continue to discussion in this way.  At last, for the finite positive value $s:=|c_{1}|+\ldots+|c_{k}|,$ we obtained that
\begin{eqnarray}\label{E2.10}&&\sum_{j=1}^{n} n(r, \frac{1}{w-d_{j}})\\\nonumber
&\leq& \max_{j=1}^{n}\{\frac{t}{(1+2+4+\ldots+2^{k})l_{j}t}\} n(r+s, w)+O(1)\\\nonumber
&=&(n\cdot \tau)n(r+s, w)+O(1),\end{eqnarray}
where $\tau\in (0, 1)$ in Lemma \ref{L6}. Then by Lemma \ref{L6}, we obtain again that $\limsup_{r\rightarrow\infty}\frac{\log T(r, w)}{r}>0.$ This is a contradiction.\par

Therefore, we obtain that $k+2\geq\deg_{w}(P)=\deg_{w}(Q)+1.$ \par

(ii). Assume that $Q(z, w(z))$ has at least one non-zero root and all its roots are non-zero. We only need modify the proof of \eqref{E2.1}. Notice that in this case,  it follow from \eqref{E2} that all zeros of $w(z)$ are not poles of $R(z, w(z)),$ and thus are not poles of $\frac{w'(z)}{w(z)}-\sum_{j=1}^{k}b_{j}(z)w(z-c_{j}).$ This implies that all the poles of $\frac{w'(z)}{w(z)}-\sum_{j=1}^{k}b_{j}(z)w(z-c_{j})$ appear only at poles of $w(z),$ or poles of $\{w(z-c_{j})\}_{j=1}^{k},$ or poles of $\{b_{j}(z)\}_{j=1}^{k}.$ Note that $T(r, b_{j}(z))=o(T(r, w(z)))$ for all $j=1, \ldots, k.$
Hence we have \begin{eqnarray*}
N(r, R(z, w(z)))&=& N\left(r, \frac{w'(z)}{w(z)}-\sum_{j=1}^{k}b_{j}(z)w(z-c_{j})\right)\\
&\leq&  N(r, w(z))+\sum_{j=1}^{k}N(r, w(z-c_{j}))+o(T(r, w)).
\end{eqnarray*}
Thus it follows from the first main theorem, Lemma \ref{L1}, Lemma \ref{L0} and Lemma \ref{L2} that
\begin{eqnarray*}&&
T(r, R(z, w(z)))= m(r, R(z, w(z)))+N(r, R(z, w(z)))\\
&\leq&m\left(r, \frac{w'(z)}{w(z)}-\sum_{j=1}^{k}b_{j}(z)w(z-c_{j})\right)+N(r, w(z))+\sum_{j=1}^{k}N(r, w(z-c_{j}))\\&&+o(T(r, w))\\
&\leq& \sum_{j=1}^{k}m(r, w(z-c_{j}))+N(r, w(z))+\sum_{j=1}^{k}N(r+|c_{j}|, w(z))+o(T(r, w))\\\nonumber
&\leq& T(r, w(z))+\sum_{j=1}^{k}\left(m(r, \frac{w(z-c_{j})}{w(z)})+m(r, w(z))\right)+\sum_{j=1}^{k}N(r, w(z))\\&&+o(T(r, w))\\
&\leq& (k+1)T(r, w(z))+o(T(r, w))
\end{eqnarray*}for all $r\not\in E,$ where $E$ is a set with zero upper density measure $E.$ On the other hand, we get from Valiron-Mohon'ko theorem (see for example in \cite{laine}) that
$$T(r, R(z, w(z)))=\deg_{w}(R)T(r, w)+o(T(r, w)).$$
Therefore, $\deg_{w}(R)\leq k+1.$ Combining with the conclusion of (i), we have $$\deg_{w}(P)=\deg_{w}(Q)+1\leq k+1.$$ \par

(iii). Suppose that $\deg_{w}(Q)=0$ and thus $R(z, w(z))$ is a polynomial. Without loss of generality, we may assume that $R(z, w(z))$ is just $P(z, w(z)),.$ Then we rewrite \eqref{E2} to be
\begin{eqnarray}\label{E2.13} \frac{w'(z)}{w(z)}+\sum_{j=1}^{k}b_{j}(z)w(z-c_{j})=P(z, w(z)).\end{eqnarray}
Suppose that $\deg_{w}(P)\geq 2.$ Then $k+2\geq\deg_{w}(P)\geq 2.$ \par

Assume that $w$ has infinitely many zeros or poles (or both) such that $N(r, \frac{1}{w})\neq o(T(r, w))$ or $N(r, w)\neq o(T(r, w))$ respectively. Without loss of generality, let $z_{0}$ be a generic pole (or a zero, respectively) of $w(z)$ with multiplicity $t,$ and thus should be a simple pole of $\frac{w'(z)}{w(z)}.$ Then by \eqref{E2.13}, it follows that $z_{0}$ is a pole (or a zero, respectively) of $P(z, w(z))$ with multiplicity $t\deg_{w}(P)(\geq 2t),$ and thus at least one term of $w(z-c_{j}),$ say $w(z-c_{1}),$ has a pole (or a zero, respectively) at $z_{0}$ with multiplicity at least $t\deg_{w}(P).$
Then shifting \eqref{E2.13} with $-c_{1}$ gives
\begin{eqnarray}\label{E2.14} \frac{w'(z-c_{1})}{w(z-c_{1})}+\sum_{j=1}^{k}b_{j}(z-c_{1})w(z-c_{j}-c_{1})=P(z-c_{1}, w(z-c_{1})).\end{eqnarray}
It follows from \eqref{E2.14} that $z_{0}$ is a pole (or a zero, respectively) of $P(z-c_{1}, w(z-c_{1}))$ with multiplicity $t\deg_{w}^{2}(P)(\geq 2^{2}t),$ and thus at least one term of $w(z-c_{j}-c_{1}),$  for convenience we say $w(z-2c_{1}),$ has a pole (or a zero, respectively) at $z_{0}$ with multiplicity at least $t\deg_{w}^{2}(P).$
Then shifting \eqref{E2.14} with $-c_{1}$ gives
\begin{eqnarray}\label{E2.15} && \frac{w'(z-2c_{1})}{w(z-2c_{1})}+\sum_{j=1}^{k}b_{j}(z-2c_{1})w(z-c_{j}-2c_{1})\\\nonumber&=&P(z-2c_{1}, w(z-2c_{1})).\end{eqnarray}
It follows from \eqref{E2.15} that $z_{0}$ is a pole (or a zero, respectively) of $P(z-2c_{1}, w(z-2c_{1}))$ with multiplicity $t\deg_{w}^{3}(P),$ and thus at least one term of $w(z-c_{j}-2c_{1}),$ for convenience we say $w(z-3c_{1}),$ has a pole (or a zero, respectively) at $z_{0}$ with multiplicity at least $t\deg_{w}^{3}(P).$ Continue to discuss by the way, at last we obtain that either
\begin{equation}\label{E2.17}
n(|z_{0}|+d|c_{1}|, w)\geq t\deg_{w}^{d}(P)
\end{equation} or
\begin{equation}\label{E2.18}
n(|z_{0}|+d|c_{1}|, \frac{1}{w})\geq t\deg_{w}^{d}(P)
\end{equation}
holds for all positive integer $d.$ It follows from the above inequalities, say \eqref{E2.17}, that
\begin{eqnarray}&&
\limsup_{r\rightarrow\infty}\frac{\log T(r, w)}{r}\\\nonumber&\geq&
\limsup_{r\rightarrow\infty}\frac{\log N(r, w)}{r}\\\nonumber&=&
\limsup_{d\rightarrow\infty}\frac{\log N(|z_{0}|+d|c_{1}|, w)}{|z_{0}|+d|c_{1}|}\\\nonumber&\geq&
\limsup_{d\rightarrow\infty}\frac{\log\left(t\deg_{w}^{d}(P)\right)+\log\log\left(|z_{0}|+d|c_{1}|\right)}{|z_{0}|+d|c_{1}|}\\\nonumber
&=&\frac{\log \deg_{w}(P)}{|c_{1}|}\\\nonumber
&>&0.
\end{eqnarray}
This contradicts to the assumption of the growth of $w.$\par

Therefore,  $w$ must have finite many zeros and poles, or have infinitely many zeros and poles such that $N(r, \frac{1}{w})=o(T(r, w))$ and $N(r, w)=o(T(r, w))$ respectively. Then by the Weierstrass (Hadamard) factorization theorem of entire functions (see for examples, \cite{goldberg-ostrovskii, yi-yang}), we may assume that \begin{equation}\label{E2.19}w(z)=f(z)e^{g(z)},\end{equation}
where $f$ is a meromorphic function such that \begin{eqnarray}\label{E2.20} T(r, f)+O(1)=\max\{N(r, \frac{1}{f}),\,\,\,\,\, N(r, f)\}=o(T(r, w)),\end{eqnarray} and $g$ is a nonconstant entire function such that \begin{eqnarray}\label{E2.21} T(r, e^{g})=T(r, w).\end{eqnarray} Submitting \eqref{E2.19} into \eqref{E2.13}, we get that
\begin{eqnarray*}\left[\frac{f'(z)}{f(z)}+g'(z)\right]+\sum_{j=1}^{k}b_{j}(z)e^{g(z-c_{j})}=P(z, f(z)e^{g(z)}),\end{eqnarray*}
and thus
\begin{eqnarray}\label{E2.22}\,\,\left[\frac{f'(z)}{f(z)}+g'(z)\right]+\left[\sum_{j=1}^{k}b_{j}(z)e^{g(z-c_{j})-g(z)}\right]e^{g(z)}=P(z, f(z)e^{g(z)}).\end{eqnarray}
Now it follows from Lemma \ref{L1}, Lemma \ref{L2}, \eqref{E2.20} and \eqref{E2.21} that
\begin{eqnarray*}\max\left\{T\left(r, \frac{f'(z)}{f(z)}+g'(z)\right), T\left(r, \sum_{j=1}^{k}b_{j}(z)e^{g(z-c_{j})-g(z)}\right)\right\}=o(T(r, w)).
\end{eqnarray*}Hence, by taking Nevanlinna characteristic from both sides of \eqref{E2.22}, we derive
\begin{eqnarray*}\deg_{w}(P)T(r, w)=T(r, P(z, f(z)e^{g(z)})=T(r, w)+o(T(r, w)).\end{eqnarray*} This contradicts to the assumption of $\deg_{w}(P)\geq 2.$\par

Therefore, it should be $\deg_{w}(P)\leq 1.$

\end{proof}

\section{Proof of Theorem \ref{T1}}\label{S3}
We first introduce two lemmas in complex analysis before proving Theorem \ref{T1}.\par

\begin{lemma}\cite[Theorem 1.50]{yi-yang}\label{L3}
Suppose that $f_{1}, f_{2}, \ldots, f_{n}$ are meromorphic functions and that $g_{1}, g_{2}, \ldots, g_{n}$ are entire functions satisfying the following conditions.\par

(i) $\sum_{j=1}^{n} f_{j}(z)e^{g_{j}(z)}\equiv 0;$\par

(ii) $g_{j}(z)-g_{k}(z)$ are not constants for $1\leq j<k\leq n;$\par

(iii) for $1\leq j\leq n,$ $1\leq h<k\leq n,$
$$T(r, f_{j})=o(T(r, e^{g_{h}-g_{k}})) \quad (r\rightarrow \infty, r\not\in E),$$
where $E\subset (1, \infty)$ is of finite linear measure or finite logarithmic measure.\\
Then $f_{j}(z)\equiv 0$ for all $j=1,2,\ldots, n.$
\end{lemma}

\begin{lemma}\cite[Theorem 1.6 of Charpter 2]{goldberg-ostrovskii} \label{L4} Let $f(z)$ be a meromorphic function, and let $f_{1}=f(az+b),$ $a\neq 0.$ Then $f(z)$ and $f_{1}(z),$ as well as $N(r, f)$ and $N(r, f_{1})$ are of the same growth category.
\end{lemma}

\begin{proof}[Proof of Theorem \ref{T1}]\par

Step 1. Obviously, an admissible solution of equation \eqref{E1.1} implies that it can not be a constant. Assume that $w$ is a nonconstant polynomial with degree $\deg w=n\geq 1,$ then $w'$ is a polynomial with degree $n-1.$ In this case, the coefficient functions $a_{0}, a_{1}$ and $b_{1}, \ldots, b_{k}$ are constants, and not all identically equal to zero. Then the right side of \eqref{E1.1} is a polynomial with degree at least $n.$ This is a contradiction.\par

Step 2. Now assume that an admissible solution $w$ is a transcendental entire function such that $\overline{N}(r, \frac{1}{w})=O(N(r, \frac{1}{w})).$  Suppose that $z_{0}$ is a zero of $w$ with multiplicity $s\geq 1.$ Then $z_{0}$ is a zero of $w'$ with multiplicity $s-1,$ and thus is a simple pole of $$k(z):=a_{0}(z)+a_{1}(z)w (z)+\sum_{j=1}^{k}b_{j}(z)w(z-c_{j})$$ according to \eqref{E1.1}. Since $w$ is entire , it is obvious that $z_{0}$ must be a pole of at least one of the coefficients $a_{0}, a_{1}$ and $b_{1}, \ldots, b_{k}.$ If the solution $w$ has infinitely many zeros $z_{0}$ such that $N(r, \frac{1}{w})\neq o(T(r, w)),$  then at least one coefficient function $c\in\{a_{0}, a_{1}, b_{1}, \ldots, b_{k}\}$ must have infinitely many poles such that $$T(r, c)\geq N(r, c)\geq \overline{N}(r, \frac{1}{w})=O(N(r, \frac{1}{w}))\neq o(T(r, w)).$$ This is a contradiction with the condition that all coefficient functions $a_{0}, a_{1}$ and $b_{1}, \ldots, b_{k}$ are small with respect to the admissible solution $w.$ Hence we have the claim that $w$ must have either at most finitely many zeros, or infinitely many zeros such that  $N(r, \frac{1}{w})=o(T(r, w)).$\par

Therefore, by the Weierstrass (Hadamard) factorization theorem of entire functions (see for examples, \cite{goldberg-ostrovskii, yi-yang}), we may assume that \begin{equation}\label{E1.2}w(z)=p(z)e^{g(z)},\end{equation}
where $p$ is either a nonzero polynomial function (thus, $T(r, p)=o(T(r, w))$), or the canonical product of all infinitely many zeros of $w$ such that $$T(r, p)+O(1)=N(r, \frac{1}{p})=N(r, \frac{1}{w})=o(T(r, w)),$$ and $g$ is a nonconstant entire function such that $T(r, e^{g})=T(r, w).$ Thus $e^{g}$ and $w$ have the same growth category.  Since $g$ is not a constant, none of $g(z-c_{j})$ is a constant.\par

We claim that\par
(i). $T(r, \frac{w(z)}{w(z-c_{j})})\neq o(T(r, w))$ for all $j=1, \ldots, k,$ and \par
(ii). $T(r, \frac{w(z-c_{n})}{w(z-c_{m})})\neq o(T(r, w))$ for any two distinct $m, n$ of $\{1, \ldots, k\}.$\\ Otherwise, there exists one $m\in\{1, \ldots, k\}$ satisfies that either $w(z-c_{j})\equiv \alpha(z) w(z)$ or $w(z-c_{m})\equiv\alpha(z) w(z-c_{n}),$ where $T(r,\alpha)=o(T(r, w))$ and $n\in\{1, \ldots, k\}\setminus\{m\}.$ This contradicts the assumption that \eqref{E1.1} has reduced form.\par

From the claim,  we can see that
all $g(z-c_{j})-g(z)$ for $1\leq j\leq k$ are not constant, and that $g(z-c_{m})-g(z-c_{n})$ are not constant for any $\{m, n\}\subset\{1, 2, \ldots, k\}.$ Otherwise, there would exist a  contradiction with the assumption (i) or (ii).\par

Submitting \eqref{E1.2} into \eqref{E1.1}, we obtain
\begin{eqnarray}\label{E1.3}
&&\left[a_{0}(z)-\frac{p'(z)+p(z)g'(z)}{p(z)}\right]e^{0}+a_{1}(z)p(z)e^{g(z)}\\\nonumber&&+\sum_{j=1}^{k}b_{j}(z)p(z-c_{j})e^{g(z-c_{j})}\equiv 0.
\end{eqnarray}  It follows from Lemma \ref{L4} that both $w(z)$ and $w(z-c_{j}),$ and thus both $e^{g(z)}$ and $e^{g(z-c_{j})}$ have the same growth category for all $j\in\{1, 2, \ldots, k\}.$ Denote $g_{k+1}:=0,$ $g_{k+2}:=g(z),$ and $g_{j}:=g(z-c_{j})$ for $j=1, \ldots, k.$ Then  for $1\leq j\leq k,$ $1\leq h<n\leq k+2,$ we have  \begin{eqnarray*}
            && \max\left\{T\left(r, a_{0}(z)-\frac{p'(z)+p(z)g'(z)}{p(z)}\right), T(r, a_{1}(z)p(z)),  T(r, b_{j}(z)p(z-c_{j}))\right\}\\&&= o(T(r, e^{g_{h}-g_{n}})).
              \end{eqnarray*}
Now it follows from Lemma \ref{L3} that $$a_{0}(z)-\frac{p'(z)+p(z)g'(z)}{p(z)}\equiv a_{1}(z)p(z)\equiv b_{j}(z)p(z-c_{j})\equiv 0,$$ which implies
$$b_{1}\equiv \cdots \equiv b_{k}\equiv a_{1}\equiv 0.$$ This is a contradiction to the assumption that none of $b_{1}, \ldots, b_{k}$ is identically equal to zero.\par

\end{proof}

\section{Proof of Theorem \ref{T2}}\label{S4}
To prove Theorem \ref{T2}, we need the following two lemmas. The first lemma below is the well-known lemma of logarithmic derivative for meromorphic functions in Nevanlinna theory.\par

\begin{lemma}\cite{hayman-1}\label{L1}
Let $f$ be a nonconstant meromorphic function. Then
\begin{eqnarray*}
m\left(r, \frac{f'(z)}{f(z)}\right)=O\left(\log rT(r, f)\right)
\end{eqnarray*} for all $r\not\in E,$ where $E$ is a set with finite logarithmic measure $E,$ i.e.,
$$\limsup_{r\rightarrow\infty}\int_{E\cap [1,r]}\frac{dt}{t}<\infty.$$
\end{lemma}

The second lemma is given by Goldberg and Ostrovskii \cite{goldberg-ostrovskii}.\par

\begin{lemma}\cite[Remark and proof of Theorem 1.6 in Charpter 2]{goldberg-ostrovskii} \label{L5} Let $f(z)$ be a meromorphic function, and let $a\in\mathbb{C}\cup\{\infty\}.$ Then
\begin{eqnarray*}T(r, f(z+b))&\geq&(1+o(1))T(r-|b|, f(z))\\&\geq& (1+o(1))T(r-2|b|, f(z+b))\end{eqnarray*}and
\begin{eqnarray*}N(r, f(z+b)=a)&\geq& (1+o(1))N(r-|b|, f(z)=a)\\&\geq& (1+o(1))N(r-2|b|, f(z+b)=a).\end{eqnarray*}
\end{lemma}

\begin{proof}[Proof of Theorem \ref{T2}]

Suppose $w$ is an admissible meromorphic solution of \eqref{E1.4}. Then obviously, $w$ is not a constant. Equation \eqref{E1.1} can be rewritten as
\begin{eqnarray}\label{E2.0}
w(z-c)=\frac{1}{b(z)}\left[\frac{w'(z)}{w(z)}-a(z)\right].
\end{eqnarray}
By the logarithmic derivative lemma (Lemma \ref{L1}) and \eqref{E2.0}, we get that
\begin{eqnarray*}m(r, w(z-c))&\leq& m(r, \frac{1}{b(z)})+m(r, \frac{w'(z)}{w(z)})+m(m(r, a(z)))+O(1)\\
&=&o(T(r, w(z))).\end{eqnarray*}
for all $r\in (1, +\infty)$ possibly outside a set $E$  with finite logarithmic measure. Hence it follows from the first main theorem and Lemma \ref{L5} that
\begin{eqnarray*}
(1+o(1))N(r+|c|, w(z))
&\geq&
N(r, w(z-c))\\&=&T(r, w(z-c))-m(r, w(z-c))\\
&=&T(r,w(z-c))+o(T(r, w))\\
&\geq&(1+o(1))T(r-|c|,w(z))+o(T(r, w))
\end{eqnarray*} for $r\not\in E.$ This implies that $N(r, w)$ and $T(r, w)$ have the same growth category.
\end{proof}



\end{document}